\documentclass[12pt]{amsart}
\usepackage{amsmath}
\usepackage{amssymb}
\usepackage{amscd}

\newtheorem{theorem}{Theorem}[section]
\theoremstyle{plain}

\newtheorem{conjecture}{Conjecture}[section]
\newtheorem{corollary}{Corollary}[section]

\newtheorem{definition}{Definition}[section]

\newtheorem{lemma}{Lemma}[section]

\newtheorem{proposition}{Proposition}[section]
\newtheorem{remark}{Remark}[section]

\numberwithin{equation}{section}

\theoremstyle{definition}
\theoremstyle{remark}
\numberwithin{equation}{section}

\begin{document}
\title[The CR $Q$-curvature flow]{Global existence and convergence for the
CR $Q$-curvature flow in a closed strictly pseudoconvex CR $3$-manifold}
\author{$^{\ast }$Shu-Cheng Chang$^{1}$}
\address{$^{1}$Department of Mathematics, National Taiwan University, Taipei
10617, Taiwan }
\email{scchang@math.ntu.edu.tw }
\thanks{$^{\ast }$Research supported in part by the MOST of Taiwan}
\author{$^{\ast }$Ting-Jung Kuo$^{2}$}
\address{$^{2}$Department of Mathematics, National Taiwan Normal University,
Taipei 11677, Taiwan }
\email{tjkuo1215@ntnu.edu.tw}
\author{Takanari Saotome$^{3}$}
\address{$^{3}$National Center for Theoretical Sciences, National Taiwan
University, Taipei 10617, Taiwan, R.O.C.}
\email{tksaotome@gmail.com}

\begin{abstract}
In this note, we affirm the partial answer to the long open Conjecture which
states that any closed embeddable strictly pseudoconvex CR $3$-manifold
admits a contact form $\theta $ with the vanishing CR $Q$-curvature. More
precisely, we deform the contact form according to an CR analogue of $Q$%
-curvature flow in a closed strictly pseudoconvex CR $3$-manifold $(M,\ J,\
[\theta _{0}])$ of the vanishing first Chern class $c_{1}(T_{1,0}M)$.
Suppose that $M$ is embeddable and the CR Paneitz operator $P_{0}$ is
nonnegative with kernel consisting of the CR pluriharmonic functions. We
show that the solution of CR $Q$-curvature flow exists for all time and has
smoothly asymptotic convergence on $M\times \lbrack 0,\infty ).$\ As a
consequence, we are able to affirm the Conjecture in a closed strictly
pseudoconvex CR $3$-manifold of the vanishing first Chern class and
vanishing torsion.
\end{abstract}

\maketitle

\section{\textbf{Introduction}}

A strictly pseudoconvex CR $(2n+1)$-manifold is called pseudo-Einstein if
the pseudohermitian Ricci curvature tensor is function-proportional to its
Levi metric for $n\geq 2$ (This is trivial for $n=1$). In general, this is
equivalent to saying that the following quantity is vanishing (\cite{l1},
\cite{h})
\begin{equation}
W_{\alpha }\doteqdot \left( R,_{\alpha }-inA_{\alpha \beta },^{\beta
}\right) =0\text{.}  \label{2019x}
\end{equation}%
Then the pseudo-Einstein condition can be replaced by (\ref{2019x}) for all $%
n\geq 1.$ From this, one can define (\cite{h}, \cite{fh}, \cite{ccc}) the CR
$Q$-curvature for $n=1$ by

\begin{equation}
Q:=-\frac{4}{3}\mathrm{Re}[\left( R,_{1}-iA_{11,\overline{1}}\right) _{%
\overline{1}}]=-\frac{2}{3}(W_{1\overline{1}}+W_{\overline{1}1}).  \label{c}
\end{equation}

J. Lee (\cite{l1}) showed an obstruction to the existence of a
pseudo-Einstein contact form $\theta ,$ which is the vanishing of first
Chern class $c_{1}(T_{1,0}M)$ for a closed strictly pseudoconvex $(2n+1)$%
-manifold $\left( M,J,\theta \right) $. Moreover, there is an invariant
contact form (\cite{fh}) if and only if it is pseudo-Einstein (\cite{l1})
for all $n\geq 1.$

\begin{conjecture}
\label{conj1} (\cite{l1}) Any closed strictly pseudoconvex CR $(2n+1)$%
-manifold of the vanishing first Chern class $c_{1}(T_{1,0}M)$ admits a
global pseudo-Einstein structure.
\end{conjecture}

Observe that the CR $Q$-curvature is vanishing when it is pseudo-Einstein.
However, it is unknown whether there is any obstruction to the existence of
a contact form $\theta $ of vanishing CR $Q$-curvature (\cite{fh}, \cite{ccc}%
). Denote a strictly pseudoconvex CR $3$-manifold by $(M,J,[\theta _{0}])$
with its contact class (or equivalently its underlying contact bundle)
indicated. For a real-valued function $\lambda $ in $M^{3}$, we consider a
conformal transformation
\begin{equation*}
\theta =e^{2\lambda }\theta _{0}.
\end{equation*}
Under this conformal change, it is known that we have the following
transformation laws for the CR $Q$-curvature:%
\begin{equation}
Q=e^{-4\lambda }(Q_{0}+2P_{0}\lambda ),  \label{0b}
\end{equation}%
here $Q_{0}$ and $P_{0}$ denote the CR $Q$-curvature and CR Paneitz operator
(see \ref{d1} for the definition) with respect to $(M,J,\theta _{0})$.

Now, we have the following conjecture.

\begin{conjecture}
\label{conj2} (\cite{fh}) Let $(M,J,\theta _{0})$ be a closed strictly
pseudoconvex CR $3$-manifold. There exists a contact form $\theta $ in the
conformal contact class $[\theta _{0}]$ with the vanishing CR $Q$-curvature.
\end{conjecture}

For a closed embeddable strictly pseudoconvex CR $3$-manifold, we decompose
\begin{equation*}
Q_{0}=(Q_{0})_{\ker }\oplus Q_{0}^{\perp }
\end{equation*}%
with respect to $P_{0}$, where $Q_{0}^{\perp }$ denotes the component of $%
Q_{0}$ which is perpendicular to $\mathrm{\ker }P_{0}$ and $(Q_{0})_{\ker }$
denotes the component of $Q_{0}$ in $\mathrm{\ker }P_{0}$. From (\ref{0b}),
we observe that if $Q=0$ for $\theta =e^{2\lambda }\theta _{0}$, then
\begin{equation*}
0=Q_{0}+2P_{0}\lambda =(Q_{0})_{\ker }\oplus (Q_{0}^{\perp }+2P_{0}\lambda
)\in (\ker P_{0})\oplus (\mathrm{\ker }P_{0})^{\perp }.
\end{equation*}%
Hence the kernel part of the background CR $Q$-curvature must be vanishing :
\begin{equation}
(Q_{0})_{\ker }=0.  \label{2020}
\end{equation}

To deal with the perpendicular part, our method is to investigating the CR $Q
$-curvature flow which is a fourth-order flow in a closed strictly
pseudoconvex CR $3$-manifold $(M,J,[\theta _{0}])$:%
\begin{equation}
\left\{
\begin{array}{l}
\frac{\partial \lambda }{\partial t}=-(Q_{0}^{\bot }+2P_{0}\lambda )+r \\
\theta (\cdot ,t)=e^{2\lambda (\cdot ,t)}\theta _{0}\text{ } \\
\lambda (\cdot ,0)=\lambda _{0}(\cdot )\text{ and }\int_{M}e^{4\lambda
_{0}}d\mu _{0}=\int_{M}d\mu _{0}%
\end{array}%
\right. ,  \label{2}
\end{equation}%
where $r=\frac{\int_{M}(Q_{0}^{\bot }+2P_{0}\lambda )d\mu }{\int_{M}d\mu }$
and $d\mu =\theta \wedge d\theta $ is the volume form with respect to $%
\theta $.

It is easy to see that the volume $V(\theta )=\int_{M}d\mu $ is invariant
under the flow (\ref{2}). Define the the functional $\mathcal{E}$ on a
closed strictly pseudoconvex CR $3$-manifold $(M,J,[\theta _{0}])$ by

\begin{equation}
\mathcal{E}\emph{(}\theta \emph{)}=\int_{M}P_{0}\lambda \cdot \lambda d\mu
_{0}+\int_{M}Q_{0}\lambda d\mu _{0}\text{, }\theta \in \lbrack \theta _{0}].
\label{1a}
\end{equation}%
It can be proved that the flow (\ref{2}) is the (volume normalized) negative
gradient flow of $\mathcal{E}\emph{(}\theta \emph{).}$ That is (see Lemma %
\ref{l3.3})%
\begin{equation*}
\frac{d}{dt}\mathcal{E}\emph{(}\theta \emph{)}=-\int_{M}(Q_{0}^{\bot
}+2P_{0}\lambda )^{2}d\mu _{0}.
\end{equation*}

In this paper, by applying the CR $Q$-curvature flow (\ref{2}) to the
functional $\mathcal{E}\emph{(}\theta \emph{)}$, we affirm the partial
answer to Conjecture \ref{conj2} in a closed embeddable strictly
pseudoconvex CR $3$-manifold $(M,\ J,\ [\theta _{0}])$ of the vanishing
first Chern class $c_{1}(T_{1,0}M)$. Here we remark that the vanishing of $%
c_{1}(T_{1,0}M)$ is required to ensure that the background assumption (\ref%
{2020}) holds for the flow (\ref{2}) as in Theorem \ref{t32}. Therefore the
limit contact form $\theta _{\infty }=e^{2\lambda _{\infty }}\theta _{0}$
has the vanishing CR $Q_{\infty }$-curvature, namely, $Q_{\infty }=0.$

Concerning global existence and a priori estimates for the solution of (\ref%
{2}), we introduced the following notions.

\begin{definition}
\label{d1}(\cite{l1}) Let $(M,J,\theta )$ be a closed strictly pseudoconvex
CR $3$-manifold. We define
\begin{equation*}
(P_{1}\varphi )\theta ^{1}=(\varphi _{\bar{1}}{}^{\bar{1}}{}_{1}+iA_{11}%
\varphi ^{1})\theta ^{1}\text{,}\varphi \in C^{\infty }(M,\mathbb{R})
\end{equation*}%
which is an operator that characterizes CR-pluriharmonic functions. Here $%
P_{1}\varphi =\varphi _{\bar{1}}{}^{\bar{1}}{}_{1}+iA_{11}\varphi ^{1}$. The
CR Paneitz operator $P$ with respect to $\theta $ is defined by%
\begin{equation}
P\varphi =2\left( \delta _{b}((P_{1}\varphi )\theta ^{1})+\overline{\delta }%
_{b}((\overline{P}_{1}\varphi )\theta ^{1})\right) ,  \label{id9}
\end{equation}%
where $\delta _{b}$ is the divergence operator that takes $(1,0)$-forms to
functions by $\delta _{b}(\sigma _{1}\theta ^{1})=\sigma _{1,}{}^{1}$, and
similarly, $\bar{\delta}_{b}(\sigma _{\bar{1}}\theta ^{\bar{1}})=\sigma _{%
\bar{1},}{}^{\bar{1}}$.
\end{definition}

The CR Paneitz operator $P$ is self-adjoint, that is, $\left\langle P\varphi
,\psi \right\rangle _{J,\theta }=\left\langle \varphi ,P\psi \right\rangle
_{J,\theta }$ for all smooth functions $\varphi $ and $\psi $. For the
details about these operators, the reader can make reference to \cite{gl},
\cite{l1}, \cite{gg}. On a closed strictly pseudoconvex CR $3$-manifold $%
(M,J,\theta ),$ we call the CR Paneitz operator $P$ with respect to $%
(J,\theta )$ "nonnegative" if
\begin{equation}
\int_{M}P\varphi \cdot \varphi \ d\mu \geq 0,  \label{4''}
\end{equation}%
for all real $C^{\infty }$ smooth functions $\varphi $, and "essentially
positive" if there exists a positive constant $\Upsilon $ $>$ $0$ such that
\begin{equation}
\int_{M}P\varphi \cdot \varphi \ d\mu \geq \Upsilon \int_{M}\varphi ^{2}\
d\mu ,  \label{4'}
\end{equation}%
for all real $C^{\infty }$ smooth functions $\varphi $ $\bot \ \ker P$ (i.e.
perpendicular to the kernel of $P$ in the $L^{2}$-norm with respect to the
volume form $d\mu $).

\begin{remark}
\label{r2} Let $(M,J,\theta )$ be a closed strictly pseudoconvex CR $3$%
-manifold.

(1) The positivity of $P$ is a CR invariant in the sense that it is
independent of the choice of the contact form $\theta $ in a fixed contact
class $[\theta _{0}]$.

(2) If $(M,J,\theta )$ is Sasakian (i.e., vanishing pseudohermitian
torsion). Then the corresponding CR Paneitz operator is essentially positive
(\cite{ccc}). Moreover, we have
\begin{equation}
\ker P_{1}=\ker P_{0}.  \label{2018C}
\end{equation}%
In general for non-embeddable CR $3$-manifolds, we only have
\begin{equation*}
\ker P_{1}\varsubsetneq \ker P_{0}.
\end{equation*}
\end{remark}

Now we are ready to state our main results for the paper.

\begin{theorem}
\label{t1} Let $(M,J,[\theta _{0}])$ be a closed embeddable strictly
pseudoconvex CR $3$-manifold of $c_{1}(T_{1,0}M)=0$. Suppose that the CR
Paneitz operator $P_{0}$ is essentially positive and its kernel is
consisting of the CR pluriharmonic functions. Then the solution of (\ref{2})
exists on $M\times \lbrack 0,\infty )$ and converges smoothly to $\lambda
_{\infty }$ $\equiv $ $\lambda (\cdot ,\infty )$ as $t$ $\rightarrow $ $%
\infty .$ Moreover, the contact form $\theta _{\infty }=e^{2\lambda _{\infty
}}\theta _{0}$ has the vanishing CR $Q$-curvature with
\begin{equation}
Q_{\infty }=0.  \label{t1a}
\end{equation}
\end{theorem}

\begin{remark}
In the paper of \ \cite{br1}, S. Brendle considered the $Q$-curvature flow
on a closed Riemannian manifold with weakly positive Paneitz operator and
its kernel consisting of constant functions that corresponds to our
assumption (\ref{2018C}). However for the CR Paneitz operator $P,$ the space
of $\ker P$ is infinite dimensional (\cite{gl}, \cite{ctw}).
\end{remark}

Note that when the pseudohermitian torsion is vanishing, then $M$ is
embeddable and the CR Paneitz operator $P_{0}$ is nonnegative (\cite{le},
\cite{ccc}). Then we are able to affirm the Conjecture \ref{conj2} in a
closed strictly pseudoconvex CR $3$-manifold of $c_{1}(T_{1,0}M)=0$ and
vanishing torsion.

\begin{corollary}
Let $(M,J,[\theta _{0}])$ be a closed strictly pseudoconvex CR $3$-manifold
of $c_{1}(T_{1,0}M)=0$ and vanishing torsion. Then the solution of (\ref{2})
exists on $M\times \lbrack 0,\infty )$ and converges smoothly to $\lambda
_{\infty }$ $\equiv $ $\lambda (\cdot ,\infty )$ as $t$ $\rightarrow $ $%
\infty .$ Moreover, the contact form $\theta _{\infty }=e^{2\lambda _{\infty
}}\theta _{0}$ has the vanishing CR $Q_{\infty }$-curvature.
\end{corollary}

For the application, we recall the CR Yamabe constant.

\begin{definition}
The CR Yamabe constant is defined by
\begin{equation*}
\sigma (M,J)=\underset{\varphi \neq 0}{\text{ }\inf }\frac{E_{\theta
}(\varphi )}{\left( \int_{M}\varphi ^{4}d\mu \right) ^{1/2}},
\end{equation*}%
where
\begin{equation*}
E_{\theta }(\varphi )=\int_{M}|\nabla _{b}\varphi |^{2}d\mu +\frac{1}{4}%
\int_{M}W\varphi ^{2}d\mu .
\end{equation*}%
Note that $E_{\widehat{\theta }}(\varphi )=E_{\theta }(u\varphi )$ for $%
\widehat{\theta }=u^{2}\theta .$ This implies that $\sigma (M,J)$ is a CR
invariant.
\end{definition}

In \cite{ccy}, they showed that under the condition of nonnegativity of $%
P_{0}$, the embeddability of $M$ (or equivalently the closedness of the
range of $\square _{b}$) follows from the positivity of the Tanaka-Webster
curvature. It is known that if $\sigma (M,J)>0$ then there always exists a
contact form $\hat{\theta}$ in $[\theta ]$ with the positive Tanaka-Webster
curvature. So we have the following application.

\begin{theorem}
\label{t3} Let $(M,J,[\theta _{0}])$ be a closed strictly pseudoconvex CR $3$%
-manifold of $c_{1}(T_{1,0}M)=0$ and the positive CR Yamabe constant.
Suppose that the CR Paneitz operator $P_{0}$ is nonnegative with kernel
consisting of the CR pluriharmonic functions. Then the solution of (\ref{2})
exists on $M\times \lbrack 0,\infty )$ and converges smoothly to $\lambda
_{\infty }$ $\equiv $ $\lambda (\cdot ,\infty )$ as $t$ $\rightarrow $ $%
\infty .$ Moreover, the contact form $\theta _{\infty }=e^{2\lambda _{\infty
}}\theta _{0}$ has the vanishing CR $Q_{\infty }$-curvature.
\end{theorem}

We briefly describe the methods used in our proofs. In section $2$, we will
briefly describe the basic notions of a pseudohermitian $3$-manifold. In
section $3,$ we derive the key Bochner-type formulae in a closed embeddable
strictly pseudoconvex CR $3$-manifold of the vanishing first Chern class $%
c_{1}(T_{1,0}M).$ In Section $4$, we discuss about an analytic property of
the CR Paneitz operator $P$ and prove the subelliptic property (\ref{2020AAA}%
) of $P_{0}$ on ($\ker P_{0})^{\perp }$. In section $5$, we show the CR $Q$%
-curvature flow (\ref{2}) is a negative gradient flow of $\mathcal{E}$.
Furthermore, we derive the uniformly $S^{4k,2}$-norm estimate for the
solution of (\ref{2}) under the subelliptic property of $P_{0}$. In section $%
6$, we prove the long time existence and asymptotic convergence of the
solution of (\ref{2}) on $M\times \lbrack 0,\infty )$ and then prove the
main results.

\section{\textbf{Preliminary}}

We introduce some basic materials in a pseudohermitian $3$-manifold (see
\cite{l1}, \cite{l2} for more details). Let $M$ be a closed $3$-manifold
with an oriented contact structure $\xi$. There always exists a global
contact form $\theta $ with $\xi =\ker \theta $, obtained by patching
together local ones with a partition of unity. The Reeb vector field of $%
\theta $ is the unique vector field $T$ such that ${\theta }(T)=1$ and $%
\mathcal{L}_{T}{\theta }=0$ or $d{\theta }(T,{\cdot })=0$. A CR structure
compatible with $\xi$ is a smooth endomorphism $J: \xi \longrightarrow \xi$
such that $J^{2}=-Id$. The CR structure $J$ can extend to $\mathbb{C}\otimes
\xi$ and decomposes $\mathbb{C} \otimes \xi$ into the direct sum of $T_{1,0}$
and $T_{0,1}$ which are eigenspaces of $J$ with respect to $i$ and $-i$,
respectively. A CR structure is called integrable, if the condition $%
[T_{1,0}, T_{1,0}] \subset T_{1,0}$ is satisfied. A pseudohermitian
structure compatible with $\xi$ is a integrable CR structure $J$ compatible
with $\xi$ together with a global contact form $\theta $.

Let $\left\{ T,Z_{1},Z_{\bar{1}}\right\} $ be a frame of $TM\otimes \mathbb{C%
}$, where $Z_{1}$ is any local frame of $T_{1,0},\ Z_{\bar{1}}=\overline{%
Z_{1}}\in T_{0,1}.$ Then $\left\{ \theta ,\theta ^{1},\theta ^{\bar{1}%
}\right\} $, the coframe dual to $\left\{ T,Z_{1},Z_{\bar{1}}\right\} $,
satisfies
\begin{equation*}
d\theta =ih_{1\bar{1}}\theta ^{1}\wedge \theta ^{\bar{1}}
\end{equation*}%
for some nonzero real function $h_{1\bar{1}}$. If $h_{1\bar{1}}$ is
positive, we call $(M,J,\theta )$ a closed strictly pseudoconvex CR $3$%
-manifold, and we can choose a $Z_{1}$ such that $h_{1\bar{1}}=1$; hence,
throughout this paper, we assume $h_{1\bar{1}}=1$.

The pseudohermitian connection of $(J,\theta)$ is the connection $\nabla$ on
$TM\otimes \mathbb{C}$ (and extended to tensors) given in terms of a local
frame $Z_1\in T_{1,0}$ by

\begin{equation*}
\nabla Z_{1}=\theta _{1}{}^{1}\otimes Z_{1},\quad \nabla Z_{\bar{1}}=\theta
_{\bar{1}}{}^{\bar{1}}\otimes Z_{\bar{1}},\quad \nabla T=0,
\end{equation*}%
where $\theta _{1}{}^{1}$ is the $1$-form uniquely determined by the
following equations:

\begin{equation}
\begin{split}
d\theta ^{1}& =\theta ^{1}\wedge \theta _{1}{}^{1}+\theta \wedge \tau ^{1} \\
\tau ^{1}& \equiv 0\mod \theta^{\bar{1}} \\
0& =\theta _{1}{}^{1}+\theta _{\bar{1}}{}^{\bar{1}},
\end{split}
\label{23}
\end{equation}%
where $\tau ^{1}$ is the pseudohermitian torsion. Put $\tau ^{1}=A^{1}{}_{%
\bar{1}}\theta ^{\bar{1}}$. The structure equation for the pseudohermitian
connection is

\begin{equation}
d\theta _{1}{}^{1}=R\theta ^{1}\wedge \theta ^{\bar{1}}+2i\mathrm{Im}(A^{%
\bar{1}}{}_{1,\bar{1}}\theta ^{1}\wedge \theta ),  \label{24}
\end{equation}%
where $R$ is the Tanaka-Webster curvature.

We will denote components of covariant derivatives with indices preceded by
comma; thus write $A^{\bar{1}}{}_{1,\bar{1}} \theta^{1} \wedge \theta $. The
indices $\{0,1,\bar{1}\}$ indicate derivatives with respect to $\{T,Z_{1},Z_{%
\bar{1}}\}$. For derivatives of a scalar function, we will often omit the
comma, for instance, $\varphi _{1}=Z_{1}\varphi ,\ \varphi _{1\bar{1}}=Z_{%
\bar{1}}Z_{1}\varphi -\theta _{1}^{1}(Z_{\bar{1}})Z_{1}\varphi ,\ \varphi
_{0}=T\varphi $ for a (smooth) function.

For a real-valued function $\varphi $, the subgradient $\nabla _{b}$ is
defined by $\nabla _{b}\varphi \in \xi \otimes \mathbb{C} = T_{1,0} \oplus
T_{0,1}$ and $-2d\theta(i Z, \nabla_{b}\varphi) =d\varphi (Z)$ for all
vector fields $Z$ tangent to contact plane. Locally $\nabla _{b}\varphi
=\varphi _{\bar{1}}Z_{1}+\varphi _{1}Z_{\bar{1}}$.

We can use the connection to define the sub-Hessian $\nabla^2_b =
(\nabla^H)^2$ as the complex linear map

\begin{equation*}
(\nabla ^{H})^{2}\varphi :T_{1,0}\oplus T_{0,1}\rightarrow T_{1,0}\oplus
T_{0,1}
\end{equation*}%
and
\begin{equation*}
(\nabla ^{H})^{2}\varphi (Z)=\nabla _{Z}\nabla _{b}\varphi .
\end{equation*}

The sub-Laplacian $\Delta _{b}$ defined as the trace of the subhessian. That
is%
\begin{equation}
\Delta _{b}\varphi =Tr\left( (\nabla ^{H})^{2}\varphi \right) =(\varphi _{1%
\bar{1}}+\varphi _{\bar{1}1}).  \label{2010b}
\end{equation}

The Levi form $\left\langle \ ,\ \right\rangle =\left\langle \ ,\
\right\rangle _{J,\theta }$ is the Hermitian form on $T_{1,0}$ defined by
\begin{equation*}
\left\langle Z,W\right\rangle _{J,\theta }=-i\left\langle d\theta ,Z\wedge
\overline{W}\right\rangle .
\end{equation*}%
We can extend $\left\langle \ ,\ \right\rangle _{J,\theta }$ to $T_{0,1}$ by
defining $\left\langle \overline{Z},\overline{W}\right\rangle _{J,\theta }=%
\overline{\left\langle Z,W\right\rangle _{J,\theta }}$ for all $Z,W\in
T_{1,0}$. The Levi form induces naturally a Hermitian form on the dual
bundle of $T_{1,0}$, and hence on all the induced tensor bundles.
Integrating the hermitian form (when acting on sections) over $M$ with
respect to the volume form $d\mu =\theta \wedge d\theta $, we get an inner
product on the space of sections of each tensor bundle. More precisely, we
denote Levi form $\left\langle \ ,\ \right\rangle _{J,\theta }$ by%
\begin{equation*}
\left\langle V,U\right\rangle _{J,\theta }=2d{\theta }(V,JU)=v_{1}u_{\bar{1}%
}+v_{\bar{1}}u_{1},
\end{equation*}%
\noindent for $V=v_{1}Z_{\bar{1}}+v_{\bar{1}}Z_{1}$,$U=u_{1}Z_{\bar{1}}+u_{%
\bar{1}}Z_{1}$ in $\xi $ and%
\begin{equation*}
(V,U)_{J,\theta }=\int_{M}\left\langle V,U\right\rangle _{J,\theta }\theta
\wedge d\theta {.}
\end{equation*}%
For a vector $X$ $\in $ $\xi ,$ we define $|X|^{2}\equiv \langle X,X\rangle
_{J,{\theta }}.$ It follows that $|\nabla _{b}\varphi |^{2}$ $=$ $2\varphi
_{1}\varphi _{\bar{1}}$ for a real valued smooth function $\varphi .$ Also
the square modulus of the sub-Hessian $\nabla _{b}^{2}\varphi $ of $\varphi $
reads $|\nabla _{b}^{2}\varphi |^{2}$ $=$ $2(\varphi _{11}\varphi _{\bar{1}%
\bar{1}}$ $+$ $\varphi _{1\bar{1}}\varphi _{\bar{1}1}).$

To consider smoothness for functions on strongly pseudo-convex manifolds, we
recall below what the Folland-Stein space $S^{k,p}$ is$.$ Let $D$ denote a
differential operator acting on functions. We say $D$ has weight $m,$
denoted $w(D)=m,$ if $m$ is the smallest integer such that $D$ can be
locally expressed as a polynomial of degree $m$ in vector fields tangent to
the contact bundle $\xi.$ We define the Folland-Stein space $S^{k,p}$ of
functions on $M$ by
\begin{equation*}
S^{k,p}=\{ \varphi \in L^{p} \mid D\varphi \in L^{p}\text{ whenever }%
w(D)\leq k\}.
\end{equation*}

\noindent We define the $L^{p}$-norm of $\nabla _{b}\varphi,$ $\nabla
_{b}^{2}\varphi$, ... to be ($\int |\nabla _{b}\varphi|^{p}\theta \wedge
d\theta )^{1/p},$ ($\int |\nabla _{b}^{2}\varphi|^{p}\theta \wedge d\theta
)^{1/p},$ $...,$ respectively, as usual. So it is natural to define the $%
S^{k,p}$-norm $||\varphi||_{S^{k,p}}$ of $\varphi\in S^{k,p}$ as follows:%
\begin{equation*}
||\varphi||_{S^{k,p}}\equiv \left(\sum_{0\leq j\leq k}||\nabla
_{b}^{j}\varphi||_{L^{p}}^{p}\right)^{1/p}.
\end{equation*}

\noindent The function space $S^{k,p}$ with the above norm is a Banach space
for $k\geq 0,$ $1<p<\infty .$ There are also embedding theorems of Sobolev
type. For instance, $S^{2,2}\subset S^{1,4}$ (for $\dim M$ $=$ $3$). We
refer the reader to, for instance, \cite{fs2} and \cite{fo} for more
discussions on these spaces.

\section{\textbf{Bochner-Type Formulae for the CR Q-Curvature}}

In this section, let $(M,J,\theta )$ be a closed embeddable strictly
pseudoconvex CR $3$-manifold of $c_{1}(T_{1,0}M)=0$. We first recall J. J.
Kohn's Hodge theory for the $\overline{\partial }_{b}$ complex (\cite{k3}).
\ Give $\eta \in \Omega ^{0,1}\left( M\right) $, a smooth $\left( 0,1\right)
$-form on $M$ with
\begin{equation*}
\overline{\partial }_{b}\eta =0,
\end{equation*}%
then there are a smooth complex-valued function $\varphi =u+iv\in C_{%
\mathbb{C}
}^{\infty }\left( M\right) $ and a smooth $\left( 0,1\right) $-form $\gamma
\in \Omega ^{0,1}\left( M\right) $ for $\gamma =\gamma _{\overline{1}}\theta
^{\overline{1}}$ such that
\begin{equation}
\left( \eta -\overline{\partial }_{b}\varphi \right) =\gamma \in \ker \left(
\square _{b}\right) ,  \label{2019C}
\end{equation}%
where $\square _{b}=2\left( \overline{\partial }_{b}\overline{\partial }%
_{b}^{\ast }+\overline{\partial }_{b}^{\ast }\overline{\partial }_{b}\right)
$ is the Kohn-Rossi Laplacian. If \ we assume $c_{1}(T_{1,0}M)=0$ and then
there is a pure imaginary $1$-form $\sigma =\sigma _{\overline{1}}\theta ^{%
\overline{1}}-\sigma _{1}\theta ^{1}+i\sigma _{0}\theta $ with
\begin{equation}
d\theta _{1}^{1}=d\sigma  \label{2019aaa}
\end{equation}%
for the pure imaginary Webster connection form $\theta _{1}^{1}.$

\begin{lemma}
\label{l31} (\cite{l1}) Let $(M,J,\theta )$ be a closed embeddable strictly
pseudoconvex CR $3$-manifold of $c_{1}(T_{1,0}M)=0$. Then there is a pure
imaginary $1$-form
\begin{equation*}
\sigma =\sigma _{\overline{1}}\theta ^{\overline{1}}-\sigma _{1}\theta
^{1}+i\sigma _{0}\theta
\end{equation*}%
with $d\theta _{1}^{1}=d\sigma $ such that%
\begin{equation}
\left\{
\begin{array}{l}
R=\sigma _{\overline{1},1}+\sigma _{1,\overline{1}}-\sigma _{0} \\
A_{11,}{}^{1}=\sigma _{1,0}+i\sigma _{0,1}-A_{11}\sigma ^{1}%
\end{array}%
\right. .  \label{16}
\end{equation}
\end{lemma}

\begin{lemma}
\label{l32} If $(M,J,\theta )$ is a closed embeddable strictly pseudoconvex
CR $3$-manifold of $c_{1}(T_{1,0}M)=0$. Then there are $u\in C_{%
\mathbb{R}
}^{\infty }\left( M\right) $ and $\gamma =\gamma _{\overline{1}}\theta ^{%
\overline{1}}\in \Omega ^{0,1}\left( M\right) $ such that%
\begin{equation}
W_{1}=2P_{1}u+i\left( A_{11}\gamma _{\overline{1}}-\gamma _{1,0}\right)
\label{01}
\end{equation}%
and
\begin{equation*}
\gamma _{\overline{1},1}=0.
\end{equation*}
\end{lemma}

\begin{proof}
By choosing
\begin{equation*}
\eta =\sigma _{\overline{1}}\theta ^{\overline{1}},
\end{equation*}%
as in (\ref{2019C}), \textbf{w}here $\sigma $ is chosen from Lemma \ref{l31},%
\textbf{\ }there exists%
\begin{equation*}
\varphi =u+iv\in C_{%
\mathbb{C}
}^{\infty }\left( M\right)
\end{equation*}

and%
\begin{equation*}
\gamma =\gamma _{\overline{1}}\theta ^{\overline{1}}\in \Omega ^{0,1}\left(
M\right) \cap \ker \left( \square _{b}\right)
\end{equation*}

such that%
\begin{equation}
\sigma _{\overline{1}}=\varphi _{\overline{1}}+\gamma _{\overline{1}}.
\label{20}
\end{equation}

Note that%
\begin{equation}
\square _{b}\gamma =0\Longrightarrow \overline{\partial }_{b}\gamma =0=%
\overline{\partial }_{b}^{\ast }\gamma \Longrightarrow \gamma _{\overline{1}%
,1}=0  \label{15}
\end{equation}

and%
\begin{equation}
\sigma _{1}=\left( \overline{\varphi }\right) _{1}+\gamma _{1}.  \label{21}
\end{equation}

Here $\gamma _{1}=\overline{\gamma _{\overline{1}}}$. From the first
equality in $\left( \ref{16}\right) $,%
\begin{equation}
R=\sigma _{\overline{1},1}+\sigma _{1,\overline{1}}-\sigma _{0}.  \label{17}
\end{equation}

Therefore%
\begin{equation*}
\begin{array}{ccl}
\sigma _{1,\overline{1}1} & = & (\overline{\varphi })_{,1\overline{1}%
1}+\gamma _{1,\overline{1}1}\text{ } \\
& = & (\overline{\varphi })_{,1\overline{1}1}\text{ } \\
& = & (\overline{\varphi })_{,\overline{1}11}+i(\overline{\varphi })_{,01}%
\text{ \ } \\
& = & (\overline{\varphi })_{,\overline{1}11}+i\left[ (\overline{\varphi }%
)_{,10}+A_{11}(\overline{\varphi })_{,\overline{1}}\right] .\text{ \ }%
\end{array}%
\end{equation*}%
The first equality due to (\ref{21}) and second equality due to (\ref{15}).
But from (\ref{15}) and (\ref{20})
\begin{equation*}
\sigma _{\overline{1},11}=\varphi _{,\overline{1}11}+\gamma _{\overline{1}%
,11}=\varphi _{,\overline{1}11}.
\end{equation*}

These and \ $\left( \ref{16}\right) ,\left( \ref{17}\right) $ imply
\begin{equation*}
\begin{array}{ccl}
W_{1} & = & \left( R,_{1}-iA_{11,\overline{1}}\right) \\
& = & \sigma _{\overline{1},11}+\sigma _{1,\overline{1}1}-i\sigma
_{1,0}+iA_{11}\sigma _{\overline{1}}\text{ } \\
& = & \varphi _{,\overline{1}11}+(\overline{\varphi })_{,\overline{1}%
11}+iA_{11}(\overline{\varphi })_{,\overline{1}}-i\gamma
_{1,0}+iA_{11}\left( \varphi _{\overline{1}}+\gamma _{\overline{1}}\right)
\\
& = & 2\left( u_{,\overline{1}11}+iA_{11}u_{\overline{1}}\right) +i\left(
A_{11}\gamma _{\overline{1}}-\gamma _{1,0}\right) \\
& = & 2P_{1}u+i\left( A_{11}\gamma _{\overline{1}}-\gamma _{1,0}\right) .%
\end{array}%
.
\end{equation*}
\end{proof}

Next we come out with the following key Bochner-type formula.

\begin{theorem}
\label{t31} Let $(M,J,\theta )$ be a closed embeddable strictly pseudoconvex
CR $3$-manifold of $c_{1}(T_{1,0}M)=0.$ Then the following equality holds
\begin{equation}
\begin{array}{l}
\int_{M}(R-\frac{1}{2}Tor)\left( \gamma ,\gamma \right) d\mu
+\int_{M}\left\vert \gamma _{1,1}\right\vert ^{2}d\mu \\
-\frac{1}{2}\int_{M}Tor^{\prime }\left( \gamma ,\gamma \right) d\mu +\frac{1%
}{4}\int_{M}(3Q+2Pu)ud\mu \\
=0.%
\end{array}
\label{2c}
\end{equation}

Here $Tor\left( \gamma ,\gamma \right) :=i(A_{\overline{1}\overline{1}%
}\gamma _{1}\gamma _{1}-A_{11}\gamma _{\overline{1}}\gamma _{\overline{1}}),$
$Tor^{\prime }\left( \gamma ,\gamma \right) :=i(A_{\overline{11},1}\gamma
_{1}-A_{11,\overline{1}}\gamma _{\overline{1}}).$
\end{theorem}

\begin{proof}
From the equality $\left( \ref{01}\right) $%
\begin{equation*}
W_{1}=2P_{1}u+in\left( A_{11}\gamma _{\overline{1}}-\gamma _{1,0}\right) ,
\end{equation*}%
we are able to get%
\begin{equation*}
\begin{array}{ccl}
\left( R,_{1}-iA_{11,\overline{1}}\right) \gamma _{\overline{1}} & = &
W_{1}\gamma _{\overline{1}} \\
& = & 2\left( u_{\overline{1}11}+iA_{11}u_{\overline{1}}\right) \gamma _{%
\overline{1}}+i(A_{11}\gamma _{\overline{1}}-\gamma _{1,0})\gamma _{%
\overline{1}} \\
& = & 2\left( u_{\overline{1}11}+iA_{11}u_{\overline{1}}\right) \gamma _{%
\overline{1}}+iA_{11}\gamma _{\overline{1}}\gamma _{\overline{1}}-\left(
\gamma _{1,1\overline{1}}-\gamma _{1,\overline{1}1}-R\gamma _{1}\right)
\gamma _{\overline{1}}.%
\end{array}%
\end{equation*}

Taking the integration over $M$ of both sides and its conjugation, we have,
by the fact that $\gamma _{1,\overline{1}}=0$,
\begin{equation}
\int_{M}\left( R-\frac{1}{2}Tor-\frac{1}{2}Tor^{\prime }\right) \left(
\gamma ,\gamma \right) d\mu +\int_{M}\left\vert \gamma _{1,1}\right\vert
^{2}d\mu -\int_{M}Tor\left( d_{b}u,\gamma \right) d\mu =0.  \label{34}
\end{equation}%
Here $Tor\left( d_{b}u,\gamma \right) =i(A_{\overline{11}\overline{\beta }%
}u_{1}\gamma _{1}-A_{11}u_{\overline{1}}\gamma _{\overline{1}}).$ On the
other hand, it follows from the equality $\left( \ref{01}\right) $ that
\begin{equation}
\left( R,_{1}-iA_{11,\overline{1}}\right) u_{\overline{1}}=W_{1}u_{\overline{%
1}}=\left[ 2P_{1}u+i\left( A_{11}\gamma _{\overline{1}}-\gamma _{1,0}\right) %
\right] u_{\overline{1}}.  \label{30A}
\end{equation}%
By the fact that $\gamma _{1,\overline{1}}=0$ again, we see that%
\begin{equation}
\begin{array}{ccl}
\int_{M}\gamma _{1,0}u_{\overline{1}}d\mu & = & -\int_{M}\gamma _{1}u_{%
\overline{1}0}d\mu \\
& = & -\int_{M}\gamma _{1}\left( u_{0\overline{1}}-A_{\overline{1}\overline{1%
}}u_{1}\right) d\mu \\
& = & \int_{M}A_{\overline{1}\overline{1}}u_{1}\gamma _{1}d\mu .%
\end{array}
\label{31A}
\end{equation}%
It follows from $\left( \ref{30A}\right) \ $and$\ \left( \ref{31A}\right) $
that%
\begin{equation*}
\begin{array}{l}
\ \ \ \frac{3}{2}\int_{M}Qud\mu +\int_{M}\left( Pu\right) ud\mu \\
=i\int_{M}\left[ \left( A_{11}u_{\overline{1}}\gamma _{\overline{1}}-A_{%
\overline{1}\overline{1}}u_{1}\gamma _{1}\right) -conj\right] d\mu \\
=-2\int_{M}Tor\left( d_{b}u,\gamma \right) d\mu .%
\end{array}%
\end{equation*}%
That is
\begin{equation}
\frac{3}{4}\int_{M}Qud\mu +\frac{1}{2}\int_{M}\left( Pu\right) ud\mu
=-\int_{M}Tor\left( d_{b}u,\gamma \right) d\mu .  \label{2020c}
\end{equation}%
Thus by (\ref{34})
\begin{equation}
\int_{M}\left( R-\frac{1}{2}Tor-\frac{1}{2}Tor^{\prime }\right) \left(
\gamma ,\gamma \right) d\mu +\int_{M}\left\vert \gamma _{1,1}\right\vert
^{2}d\mu +\frac{3}{4}\int_{M}Qud\mu +\frac{1}{2}\int_{M}\left( Pu\right)
ud\mu =0.  \label{35}
\end{equation}
\end{proof}

\begin{theorem}
\label{t32} Let $(M,J,\theta )$ be a closed embeddable strictly pseudoconvex
CR $3$-manifold of $c_{1}(T_{1,0}M)=0$ and the CR Paneitz operator $P$ with
kernel consisting of the CR pluriharmonic functions. Then
\begin{equation}
Q_{\ker }=0.  \label{2019b}
\end{equation}%
Here $Q=Q_{\ker }+Q^{\perp }$. $Q^{\perp }$ is in $(\ker P)^{\perp }$ which
is perpendicular to the kernel of self-adjoint Paneitz operator $P$ in the $%
L^{2}$ norm with respect to the volume form $d\mu $ $=$ $\theta \wedge
d\theta $.
\end{theorem}

\begin{proof}
Note that it follows from (\cite{l1}) that a real-valued smooth function $u$
is said to be CR-pluriharmonic if, for any point $x\in M$, there is a
real-valued smooth function $v$ such that
\begin{equation}
\overline{\partial }_{b}(u+iv)=0.  \label{20a}
\end{equation}%
Then the equalty (\ref{01}) still holds if we replace $\ u$ by $(u+CQ_{\ker
}).$ It follows from the Bochner-type formula (\ref{2c}) that
\begin{equation*}
\begin{array}{l}
\int_{M}(R-\frac{1}{2}Tor)\left( \gamma ,\gamma \right) d\mu -\frac{1}{2}%
\int_{M}Tor^{\prime }\left( \gamma ,\gamma \right) d\mu \\
+\int_{M}\left\vert \gamma _{1,1}\right\vert ^{2}d\mu +\frac{1}{2}%
\int_{M}\left( Pu\right) ud\mu +\frac{3}{4}\int_{M}Qud\mu +\frac{3}{4}%
C\int_{M}(Q_{\ker })^{2}d\mu \\
=0.%
\end{array}%
\end{equation*}%
However, if $\int_{M}(Q_{\ker })^{2}d\mu $ is not zero, this will lead to a
contradiction by choosing the constant $C<<-1$ or $C>>1.$ Then we are done.
\end{proof}

\section{\textbf{Subelliptic Property for the CR Paneitz Operator}}

Suppose that $(M^{2n+1},\theta ,J)$ is a closed strictly pseudoconvex CR $%
(2n+1)$-manifold. In \cite{bm}, Boutet de Monvel proved that $M$ can be
embedded in $\mathbb{C}^{N}$ for some $N$ if $n\geq 2.$ In the case of $\dim
{M}=3$, D. Burns (\cite{bu}) showed that if the range of $\overline{\partial
}_{b}$ is closed in $L^{2}(M),$ then the Boutet de Monvel's construction
works, and $M$ can be embedded.

Conversely, using a microlocalization method, J. J. Kohn proved that if $%
M^{2n+1}$ is a boundary of a bounded pseudo-convex domain $\Omega \subset
\mathbb{C}^{2n}$ then the range of $\overline{\partial }_{b}$ is closed in $%
L^{2}(M)$ (\cite{k1}, \cite{k2}). In particular, any closed pseudohermitian $%
3$-manifold with vanishing torsion has a tangential CR operator $\overline{%
\partial }_{b}$ which has a closed range in $L^{2}(M)$ (Theorem 2.1 of \cite%
{le}).

\begin{definition}
We call the \textbf{subelliptic }property of CR Paneitz operator $P$ on the
orthogonal complement of $\ker {P}$ in the sense that the following estimate
holds
\begin{equation}
\Vert \lambda \Vert _{S^{k+4,2}}^{2}\leq C_{k}(\Vert P\lambda \Vert
_{S^{k,2}}^{2}+\Vert \lambda \Vert _{L^{2}}^{2})  \label{2020AAA}
\end{equation}%
for some constant $C_{k}$ independent of $\lambda \in (\ker {P})^{\perp }$.
\end{definition}

Next, we recall that if the range of $\overline{\partial }_{b}$ is closed in
$L^{2}(M)$, we have a subelliptic estimate for $\overline{\partial }_{b}$ on
the orthogonal complement of $\ker {\square _{b}}($ \cite{k1}, \cite{k2} ).

\begin{proposition}
\label{p32} Let $(M,J,\theta )$ be a closed strictly pseudoconvex CR $3$%
-manifold. If the range of $\overline{\partial }_{b}$ is closed in $L^{2}(M)$%
, then the sublliptic property of CR Paneitz operator $P$ holds on the
orthogonal complement of $\mathrm{\ker }{P.}$
\end{proposition}

\proof
Since the Kohn Laplacian (acting on functions) satisfies $\square
_{b}=-\Delta +iT$, we have
\begin{equation*}
\Delta _{b}^{2}+T^{2}=\overline{\square }_{b}\square _{b}-i[\Delta
_{b},T]=\square _{b}\overline{\square }_{b}+i[\Delta _{b},T].
\end{equation*}%
Moreover, since it holds that $i[\Delta _{b},T]\lambda =2i(A_{\overline{1}%
\overline{1}}\lambda ^{\overline{1}})^{,\overline{1}}+2i(A_{11}\lambda
^{1})^{,1}$, the CR Paneitz operator can be written (\cite{gl}) by
\begin{equation*}
P\lambda =\overline{\square }_{b}\square _{b}\lambda -2S\lambda =\square _{b}%
\overline{\square }_{b}\lambda -2\overline{S}\lambda ,
\end{equation*}%
where $S\lambda =2i(A_{\overline{1}\overline{1}}\lambda ^{\overline{1}})^{,%
\overline{1}}.$

We observe that the kernel of $P_{1}$ is just the space of all CR
pluriharmonic functions (see definition \ref{d1}). Let $w =u+iv\in \ker
\square _{b},\ u, v\in C^{\infty }(M;\mathbb{R})$. By Lemma 3.1. of \cite{l1}%
,
\begin{equation*}
P_{1}u=P_{1}v=0\ \ \mathrm{and\ \ }\ u,v\in \ker {P}.
\end{equation*}%
Therefore if $\lambda \in (\ker {P})^{\perp }\cap C^{\infty }(M;\mathbb{R}),$
then it holds that
\begin{equation*}
(\lambda ,w )_{L^{2}}=(\lambda ,u)_{L^{2}}+i(\lambda ,v)_{L^{2}}=0.
\end{equation*}%
Since the CR Paneitz operator is a real operator, i.e. $\overline{P}%
\overline{\lambda }=\overline{P\lambda },$ we have $\ker {P}=\{\lambda
+i\eta \mid \lambda ,\eta \in \ker {P}\cap C^{\infty }(M;\mathbb{R})\}.$
This concludes that
\begin{equation*}
(\ker {P})^{\perp }\subset (\ker {\square _{b}})^{\perp }.
\end{equation*}

Let $\lambda \in (\ker {P})^{\perp }\cap C^{\infty }(M;\mathbb{R})$. Then
\begin{equation*}
\lambda \in (\ker {\square _{b}})^{\perp }\cap (\ker {\overline{\square }_{b}%
})^{\perp }
\end{equation*}%
so that by the closedness of the range of $\square _{b}$, there exists $%
\lambda ^{\prime }\in (\ker {\overline{\square }_{b}})^{\perp }\subset
C^{\infty }(M;\mathbb{C})$ such that $\lambda =\overline{\square }%
_{b}\lambda ^{\prime }.$ Since $\ker {\overline{\square }_{b}}\subset \ker {P%
},$ if $w \in \ker {\overline{\square }_{b},}$ then we have
\begin{equation*}
(\square \lambda ,w )_{L^{2}}=(\lambda ^{\prime },\overline{\square }%
_{b}\square _{b}w )_{L^{2}}=(\lambda ^{\prime },2Sw )_{L^{2}}.
\end{equation*}%
In particular%
\begin{equation}
\square _{b}\lambda -2S\lambda ^{\prime }\in (\ker {\overline{\square }_{b}}%
)^{\perp }.  \label{2011}
\end{equation}%
Moreover, since $\lambda ^{\prime }\in (\ker {\overline{\square }_{b}}%
)^{\perp }$,
\begin{equation*}
\Vert S\lambda ^{\prime }\Vert _{L^{2}}\leq C\Vert \lambda ^{\prime }\Vert
_{S^{2,2}}\leq C^{\prime }\Vert \lambda \Vert _{L^{2}}.
\end{equation*}%
Therefore
\begin{eqnarray}
\Vert \square _{b}\lambda -2S\lambda ^{\prime }\Vert _{S^{2,2}}^{2} &\leq
&2(\Vert \square _{b}\lambda \Vert _{S^{2,2}}^{2}+4\Vert S\lambda ^{\prime
}\Vert _{S^{2,2}}^{2})  \notag \\
{} &\leq &C^{\prime }(\Vert \square _{b}\lambda \Vert _{S^{2,2}}^{2}+\Vert
\lambda \Vert _{S^{2,2}}^{2})  \notag \\
{} &\leq &C^{\prime \prime }\Vert \square _{b}\lambda \Vert _{S^{2,2}}^{2}.
\end{eqnarray}%
Similarly%
\begin{equation*}
\Vert \square _{b}\lambda \Vert _{S^{2,2}}^{2}\leq C\Vert \square
_{b}\lambda -2S\lambda ^{\prime }\Vert _{S^{2,2}}^{2}.
\end{equation*}

It follow from (\ref{2011}) and the subelliptic estimates for $\square _{b},%
\overline{\square }_{b}$ on their orthogonal complement of the kernel, for $%
\lambda \in (\ker {P})^{\perp }$
\begin{eqnarray*}
\Vert \lambda \Vert _{S^{4,2}} &\leq &C_{1}\Vert \square _{b}\lambda \Vert
_{S^{2,2}}\leq C_{2}\Vert \square _{b}\lambda -2S\lambda ^{\prime }\Vert
_{S^{2,2}} \\
&\leq &C_{3}\Vert \overline{\square }_{b}(\square _{b}\lambda -2S\lambda
^{\prime })\Vert _{L^{2}} \\
&\leq &C_{4}(\Vert \overline{\square }_{b}\square _{b}\lambda \Vert
_{L^{2}}+\Vert \overline{\square }_{b}S\lambda ^{\prime }\Vert _{L^{2}}) \\
&\leq &C_{5}(\Vert \overline{\square }_{b}\square _{b}\lambda \Vert
_{L^{2}}+\Vert \lambda \Vert _{L^{2}}).
\end{eqnarray*}

Hence $\overline{\square }_{b}\square _{b}$ has the subellipticity on $(\ker
{P})^{\perp }.$ Since $P=\overline{\square }_{b}\square _{b}-2S$ and $S$ is
a lower order operator, we obtain the subelliptic property (\ref{2020AAA})
for $P$ on $(\ker {P})^{\perp }$.
\endproof%

As a consequence of Proposition \ref{p32} and results of J. J. Kohn (\cite%
{k1}, \cite{k2}), one obtains

\begin{corollary}
\label{c31} Let $(M,\theta ,J)$ be a closed embeddable strictly pseudoconvex
CR $3$-manifold. Then the subelliptic property (\ref{2020AAA}) of CR Paneitz
operator $P$ holds on the orthogonal complement of $\mathrm{\ker }{P.}$
\end{corollary}

\section{\textbf{A Priori Uniformly Estimates }}

Let $T$ \ be the maximal time for a solution of the flow (\ref{2}) on $%
M\times \lbrack 0,T)$. We will derive the uniformly $S^{4k,2}$-norm estimate
for $\lambda $ under the flow (\ref{2}) for all $t\geq 0$. It then follows
that we have the long-time existence and asymptotic convergence for
solutions of (\ref{2}) on $M\times \lbrack 0,\infty )$ as in section $5$.

First we recall a pseudohermitian Moser inequality. Let $\Vert \cdot \Vert
_{p}$ denote the $L^{p}$-norm with respect to the volume form $d\mu _{0}$.

\begin{lemma}
\label{l3.2} (\cite{ccc}) Let $(M,J,\theta _{0})$ be a closed strictly
pseudoconvex CR $3$-manifold. Then there exist constants $C,$ $\varkappa ,$
and $\nu $ such that for all $\varphi \in C^{\infty }(M)$, there holds
\begin{equation}
\displaystyle\int_{M}e^{\varphi }\ d\mu _{0}\leq C\exp {\left( \varkappa
\left\Vert \overset{\circ }{\nabla }_{b}\varphi \right\Vert _{4}^{4}+\nu
\left\Vert \varphi \right\Vert _{4}^{4}\right) .}  \label{moine}
\end{equation}
\end{lemma}

\begin{lemma}
\label{l3.3} Let $(M,J,[\theta _{0}])$ be a closed strictly pseudoconvex CR $%
3$-manifold. Let $\lambda $ be a solution of the flow (\ref{2}) on $M\times
\lbrack 0,T)$. Then there exists a positive constant $\beta =\beta
(Q_{0},\theta _{0})$ such that%
\begin{equation}
\mathcal{E}\emph{(}\theta \emph{)}=\displaystyle\int_{M}P_{0}\lambda \cdot
\lambda d\mu _{0}+\displaystyle\int_{M}Q_{0}\lambda d\mu _{0}\leq \beta ^{2},
\label{9}
\end{equation}%
for all $t\geq 0.$
\end{lemma}

\proof
\ It follow from (\ref{1b}) and (\ref{0b}) that
\begin{equation}
\begin{array}{l}
\displaystyle\frac{d}{dt}\mathcal{E}\emph{(}\theta \emph{)} \\[8pt]
=2\displaystyle\int_{M}P_{0}\lambda \cdot \frac{\partial \lambda }{\partial t%
}d\mu {0}+\displaystyle\int_{M}Q_{0}\frac{\partial \lambda }{\partial t}d\mu
_{0} \\[8pt]
=-2\displaystyle\int_{M}(Q_{0}^{\bot }+2P_{0}\lambda -r(t))P_{0}\lambda d\mu
_{0}-\displaystyle\int_{M}Q_{0}(Q_{0}^{\bot }+2P_{0}\lambda -r(t))d\mu _{0}
\\[8pt]
=-2\displaystyle\int_{M}(Q_{0}^{\bot }+2P_{0}\lambda )P_{0}\lambda d\mu _{0}-%
\displaystyle\int_{M}Q_{0}(Q_{0}^{\bot }+2P_{0}\lambda )d\mu _{0} \\[8pt]
=-2\displaystyle\int_{M}(Q_{0}^{\bot }+2P_{0}\lambda )P_{0}\lambda d\mu _{0}-%
\displaystyle\int_{M}Q_{0}^{\bot }(Q_{0}^{\bot }+2P_{0}\lambda )d\mu _{0} \\%
[8pt]
=-\displaystyle\int_{M}(Q_{0}^{\bot }+2P_{0}\lambda )^{2}d\mu _{0}.%
\end{array}
\label{2d}
\end{equation}%
\endproof%

We also observe that

\begin{lemma}
\label{ODE} Let $f$ $:[0,T)\rightarrow \mathbb{R}$ be a $C^{1}$ smooth
function satisfying%
\begin{equation*}
f^{\prime }\leq -L_{1}f+L_{2}
\end{equation*}%
for some positive constants $L_{1},$ $L_{2}>0.$ Then
\begin{equation*}
f(t)\leq f(0)e^{-L_{1}t}+\frac{L_{2}}{L_{1}}
\end{equation*}%
for $t$ $\in $ $[0,T).$
\end{lemma}

We will often use the above lemma (whose proof is left to the reader) to
obtain the higher order estimates. Now we write $\lambda =\lambda _{\ker
}+\lambda ^{\perp }.$ $Q_{0}=(Q_{0})_{\ker }+Q_{0}^{\perp }$ with respect to
$P_{0}.$ Comparing both sides of the following formula:

\begin{equation*}
\frac{\partial \lambda _{\ker }}{\partial t}+\frac{\partial \lambda ^{\perp }%
}{\partial t}=\frac{\partial \lambda }{\partial t}=-(Q_{0}^{\perp
}+2P_{0}\lambda )+r(t).
\end{equation*}%
we obtain%
\begin{equation}
\frac{\partial \lambda ^{\perp }}{\partial t}=-(Q_{0}^{\perp }+2P_{0}\lambda
^{\perp })  \label{11}
\end{equation}%
and
\begin{equation}
\frac{\partial \lambda _{\ker }}{\partial t}=r(t).  \label{11a}
\end{equation}

From now on, $C$ or $C_{j\text{ }}$ denotes a generic constant which may
vary from line to line.

\begin{proposition}
\label{p3.1} Let $(M,J,[\theta _{0}])$ be a closed strictly pseudoconvex CR $%
3$-manifold. Suppose that the CR Paneitz operator $P_{0}$ has the
subelliptic property (\ref{2020AAA}) on the orthogonal complement of $%
\mathrm{ker}{P}_{0}$ and essentially positive. \ Then under the flow (\ref{2}%
), there exists a positive constant $C(k,\Upsilon ,\theta _{0},\beta )>0$
such that%
\begin{equation*}
||\lambda ||_{S^{4k,2}}\leq C(k,\Upsilon ,\theta _{0},\beta )
\end{equation*}%
for all $t\geq 0.$
\end{proposition}

\proof
From (\ref{11a}), we have
\begin{equation}
\lambda _{\ker }(t,p)-\lambda _{\ker }(0,p)=\displaystyle \int_{0}^{t}rdt.
\label{1.6}
\end{equation}%
Since ($P_{0}$ being self-adjoint)%
\begin{equation*}
\displaystyle \int_{M}P_{0}\lambda ^{\perp }\cdot \lambda _{\ker }d\mu _{0}=0%
\text{ }
\end{equation*}%
and
\begin{equation*}
\begin{array}{l}
\displaystyle \int_{M}Q_{0}\lambda _{\ker }d\mu _{0} \\[8pt]
=\displaystyle \int_{M}Q_{0}[\lambda _{\ker }(0,p)+\displaystyle %
\int_{0}^{t}rdt]d\mu _{0} \\[8pt]
=\displaystyle \int_{M}Q_{0}\lambda _{\ker }(0,p)d\mu _{0} \\
\geq -C.%
\end{array}%
\end{equation*}%
We compute%
\begin{equation*}
\begin{array}{l}
\displaystyle \int_{M}P_{0}\lambda \cdot \lambda d\mu _{0}+\displaystyle %
\int_{M}Q_{0}\lambda d\mu _{0} \\[8pt]
=\displaystyle \int_{M}P_{0}\lambda ^{\perp }\cdot (\lambda _{\ker }+\lambda
^{\perp })d\mu _{0}+\displaystyle \int_{M}Q_{0}(\lambda _{\ker }+\lambda
^{\perp })d\mu _{0} \\[8pt]
\geq \displaystyle \int_{M}P_{0}\lambda ^{\perp }\cdot \lambda ^{\perp }d\mu
_{0}+\displaystyle \int_{M}Q_{0}\lambda ^{\perp }d\mu _{0}-C.%
\end{array}%
\end{equation*}%
It then follows from Lemma \ref{l3.3} that
\begin{equation*}
\displaystyle \int_{M}P_{0}\lambda ^{\perp }\cdot \lambda ^{\perp }d\mu _{0}+%
\displaystyle \int_{M}Q_{0}\lambda ^{\perp }d\mu _{0}\leq (\beta ^{2}+C)
\end{equation*}%
for all $t$ $\geq 0.$ Since the CR Paneitz operator $P_{0}$ is essentially
positive in the sense that
\begin{equation*}
\displaystyle \int_{M}P_{0}\lambda ^{\perp }\cdot \lambda ^{\perp }d\mu
_{0}\geq \Upsilon \displaystyle \int_{M}(\lambda ^{\perp })^{2}d\mu _{0}.
\end{equation*}%
All these imply that there exists a positive constant $C(\Upsilon
,Q_{0},\theta _{0})>0$ such that%
\begin{equation*}
\begin{array}{ccc}
(\beta ^{2}+C) & \geq & \displaystyle \int_{M}P_{0}\lambda ^{\perp }\cdot
\lambda ^{\perp }d\mu _{0}+\displaystyle \int_{M}Q_{0}\lambda ^{\perp }d\mu
_{0} \\[8pt]
& \geq & \displaystyle \frac{\Upsilon }{2}\displaystyle \int_{M}(\lambda
^{\perp })^{2}d\mu _{0}-C(\Upsilon ,Q_{0}^{\perp },\theta _{0})%
\end{array}%
\end{equation*}%
for all $t$ $\geq 0.$ So there exists a positive constant $C(\Upsilon
,Q_{0}^{\perp },\theta _{0},\beta )$ such that
\begin{equation}
\displaystyle \int_{M}(\lambda ^{\perp })^{2}d\mu _{0}\leq C(\Upsilon
,Q_{0}^{\perp },\theta _{0},\beta )  \label{1.2}
\end{equation}%
for all $t$ $\geq 0.$

Next we compute, for all positive integers $k,$
\begin{equation}
\begin{array}{l}
\displaystyle\frac{d}{dt}\displaystyle\int_{M}|P_{0}^{k}\lambda ^{\perp
}|^{2}\,d\mu _{0} \\[8pt]
=2\displaystyle\int_{M}(P_{0}^{k}\lambda ^{\perp })\left( P_{0}^{k}\frac{%
\partial \lambda ^{\perp }}{\partial t}\right) \,d\mu _{0} \\[8pt]
=-2\displaystyle\int_{M}(P_{0}^{k}\lambda ^{\perp })\,P_{0}^{k}(Q_{0}^{\perp
}+2P_{0}\lambda ^{\perp })\,d\mu _{0} \\[8pt]
=-2\displaystyle\int_{M}(P_{0}^{k}\lambda ^{\perp })(P_{0}^{k}Q_{0}^{\perp
})\,d\mu -4\displaystyle\int (P_{0}^{k}\lambda ^{\perp })(P_{0}^{k+1}\lambda
^{\perp })\,d\mu _{0}.%
\end{array}
\label{eqn1.6}
\end{equation}%
By essential positivity of $P_{0},$ we obtain%
\begin{equation}
\displaystyle\int_{M}(P_{0}^{k}\lambda ^{\perp })(P_{0}^{k+1}\lambda ^{\perp
})\,d\mu _{0}=\displaystyle\int_{M}(P_{0}^{k}\lambda ^{\perp
})(P_{0}P_{0}^{k}\lambda ^{\perp })d\mu _{0}\geq \Upsilon \displaystyle%
\int_{M}(P_{0}^{k}\lambda ^{\perp })^{2}d\mu _{0}.  \label{eqn1.7}
\end{equation}%
Therefore there exists a constant $C=C(k,Q_{0}^{\perp },\Upsilon )$ such
that
\begin{equation}
\frac{d}{dt}\displaystyle\int_{M}|P_{0}^{k}\lambda ^{\perp }|^{2}\,d\mu
_{0}\leq -3\Upsilon \displaystyle\int_{M}|P_{0}^{k}\lambda ^{\perp
}|^{2}\,d\mu _{0}+C(k,Q_{0}^{\perp },\Upsilon ).  \label{eqn1.7'}
\end{equation}%
By applying Lemma \ref{ODE} to the O.D.E. $f^{\prime }(t)\leq -3\Upsilon
f(t)+C(k)$ from (\ref{eqn1.7'}), we obtain
\begin{equation*}
\displaystyle\int_{M}|P_{0}^{k}\lambda ^{\perp }|^{2}\,d\mu _{0}\leq
C(k,Q_{0}^{\perp },\Upsilon ).
\end{equation*}%
By the subelliptic property of \ $P_{0}$ and (\ref{1.2}), we obtain
\begin{equation*}
||\lambda ^{\perp }-\overline{\lambda ^{\perp }}||_{S^{4k,2}}\leq
C(k,Q_{0}^{\perp },\Upsilon )
\end{equation*}%
for all $t\geq 0.$

We note that
\begin{equation*}
\lambda -\overline{\lambda }=(\lambda ^{\perp }-\overline{\lambda ^{\perp }}%
)+[\lambda _{\ker }(0,p)-\overline{\lambda _{\ker }(0,p)}],
\end{equation*}%
and hence
\begin{equation}
||\lambda -\overline{\lambda }||_{S^{4k,2}}\leq C(\lambda _{\ker
}(0,p),k)+||\lambda ^{\perp }-\overline{\lambda ^{\perp }}||_{S^{4k,2}}\leq
C(k,Q_{0}^{\perp },\Upsilon )  \label{eqn1.8}
\end{equation}%
for all $t\geq 0.$ Recall that the average $\bar{f}$ of a function $f$ is
defined by $\overline{f}=\frac{\int_{M}fd\mu _{0}}{\int_{M}d\mu _{0}}.$ In
particular, there holds%
\begin{equation*}
||\lambda -\overline{\lambda }||_{S^{4,2}}\leq C(Q_{0}^{\perp },\Upsilon ).
\end{equation*}%
Therefore by the Sobolev embedding theorem, we have $S^{4,2}\subset S^{1,8}$
and%
\begin{equation}
||\lambda -\overline{\lambda }||_{S^{1,8}}\leq C(Q_{0}^{\perp },\Upsilon ).
\label{1.3}
\end{equation}

Now using Lemma \ref{l3.2} (pseudohermitian Moser inequality), we get
\begin{equation*}
\displaystyle \int_{M}e^{4(\lambda -\overline{\lambda })}d\mu _{0}\leq C\exp
(C||\lambda -\overline{\lambda }||_{S^{1,4}})\leq C(Q_{0}^{\perp },\Upsilon
).
\end{equation*}%
Together with $\displaystyle \int_{M}e^{4\lambda }d\mu _{0}$ being invariant
under the flow, we conclude that%
\begin{equation}
C\geq \overline{\lambda }\geq -C(Q_{0}^{\perp },\Upsilon ).  \label{eqn1.9}
\end{equation}%
Here the upper bound is obtained by observing that $\displaystyle \int
\lambda d\mu _{0}\leq \displaystyle \int e^{4\lambda }d\mu _{0}$.

We finally obtain
\begin{equation}
||\lambda ||_{S^{4k,2}}\leq C(Q_{0}^{\perp },\Upsilon )  \label{eqn4.2}
\end{equation}%
for all $t\geq 0.$
\endproof%

\section{\textbf{Long-Time Existence and Asymptotic Convergence}}

In the previous section, we obtain an a priori $S^{4k,2}$ uniformly estimate
(i.e., independent of the time) of the solution of (\ref{2}) on $M\times
\lbrack 0,\infty )$ if the CR Paneitz operator $P_{0}$ has the subelliptic
property on the orthogonal complement of $\mathrm{ker}{P}_{0}$ and
essentially positive. Therefore we are able to prove the long time existence
and asymptotic convergence of solutions of (\ref{2}).

\begin{theorem}
\label{T61} Let $(M,J,[\theta _{0}])$ be a closed embeddable strictly
pseudoconvex CR $3$-manifold. Suppose that the CR Paneitz operator $P_{0}$
is essentially positive and has the subelliptic property on the orthogonal
complement of $\mathrm{\ker }{P}_{0}$. Then the solution of (\ref{2}) exists
on $M\times \lbrack 0,\infty )$ and converges smoothly to $\lambda _{\infty
} $ $\equiv $ $\lambda (\cdot ,\infty )$ as $t$ $\rightarrow $ $\infty .$
Moreover, the contact form $\theta _{\infty }=e^{2\lambda _{\infty }}\theta
_{0}$ has the CR $Q$-curvature with
\begin{equation}
Q_{\infty }=e^{-4\lambda _{\infty }}(Q_{0})_{\ker }.
\end{equation}%
In additional if $(Q_{0})_{\ker }=0,$ then
\begin{equation*}
Q_{\infty }=0.
\end{equation*}
\end{theorem}

\proof
Let $(M,J,[\theta _{0}])$ be a closed embeddable strictly pseudoconvex CR $3$%
-manifold. Since $P_{0}$ is nonnegative and has the subelliptic property on $%
(\ker P_{0})^{\perp }$. It follows that there exists a unique $C^{\infty }$
smooth solution $\lambda ^{\perp }$ of (\ref{11}) for a short time. On the
other hand, we can apply the contraction mapping principle to show the short
time existence of a unique $C^{\infty }$ smooth solution $\lambda _{\ker }$
to (\ref{11a}). We refer to the first named author's previous paper (\cite%
{ccc}) for some details. The long time solution then follows from Corollary %
\ref{c31}, the Sobolev embedding theorem for $S^{k,2},$ and the standard
argument for extending the solution at the maximal time $T.$

Starting from (\ref{2d}), we compute%
\begin{eqnarray}
\frac{d^{2}}{dt^{2}}\mathcal{E}(\theta ) &=&-4\displaystyle %
\int_{M}(Q_{0}^{\bot }+2P_{0}\lambda )P_{0}\frac{\partial \lambda }{\partial
t}d\mu _{0}  \label{eqn4.3} \\
&=&4\displaystyle \int_{M}(Q_{0}^{\perp }+2P_{0}\lambda )P_{0}(Q_{0}^{\perp
}+2P_{0}\lambda )d\mu _{0}\text{ }  \notag \\
&\geq &4\Upsilon \displaystyle \int_{M}(Q_{0}^{\perp }+2P_{0}\lambda
)^{2}d\mu _{0}  \notag
\end{eqnarray}%
by essential positivity of $P_{0}$. Therefore $\frac{d}{dt}\mathcal{E}%
(\theta )$ is nondecreasing, and hence%
\begin{equation}
\displaystyle \int_{M}e^{4\lambda }[Q-e^{-4\lambda }(Q_{0})_{\ker }]^{2}d\mu
\text{ (}\geq 0)\text{ is nonincreasing}.  \label{eqn4.4}
\end{equation}%
By (\ref{eqn4.2}), we can find a sequence of times $t_{j}$ such that $%
\lambda _{j}$ $\equiv $ $\lambda (\cdot ,t_{j})$ converges to $\lambda
_{\infty }$ in $C^{\infty }$ topology as $t_{j}$ $\rightarrow \infty $. On
the other hand, integrating (\ref{2d}) gives%
\begin{equation}
\mathcal{E}(\lambda _{\infty })-\mathcal{E}(\lambda _{0})=-\displaystyle %
\int_{0}^{\infty }\displaystyle \int_{M}e^{4\lambda }[Q-e^{-4\lambda
}(Q_{0})_{\ker }]^{2}d\mu dt.  \label{eqn4.5}
\end{equation}%
In view of (\ref{eqn4.4}), (\ref{eqn4.5}), we obtain

\begin{equation}
0=\lim_{t\rightarrow \infty }\displaystyle \int_{M}e^{4\lambda
}[Q-e^{-4\lambda }(Q_{0})_{\ker }]^{2}d\mu =\displaystyle %
\int_{M}e^{4\lambda _{\infty }}[Q_{\infty }-e^{-4\lambda _{\infty
}}(Q_{0})_{\ker }]^{2}d\mu _{\infty }  \label{eqn4.6}
\end{equation}%
where $Q_{\infty }$ denotes the $Q$-curvature with respect to $(J,\theta
_{\infty }),$ $\theta _{\infty }$ $=$ $e^{2\lambda _{\infty }}\theta _{0},$
and $d\mu _{\infty }$ $=$ $e^{4\lambda _{\infty }}d\mu _{0}.$ It follows that%
\begin{equation*}
Q_{\infty }=e^{-4\lambda _{\infty }}(Q_{0})_{\ker }.
\end{equation*}%
That is
\begin{equation*}
2P_{0}\lambda _{\infty }^{\perp }+Q_{0}^{\perp }=0.
\end{equation*}

In the following, we are going to prove the smooth convergence for all time.
First we want to prove that $\lambda $ converges to $\lambda _{\infty }$ in $%
L^{2}.$ Write $\lambda _{\infty }$ $=$ $\lambda _{\infty }^{\perp }+(\lambda
_{\infty })_{\ker }.$ Observe that ($||\cdot ||_{2}$ denotes the $L^{2}$
norm with respect to the volume form $d\mu _{0}$)%
\begin{eqnarray}
||\lambda _{\infty }^{\perp }-\lambda ^{\perp }||_{2} &\leq &||\lambda
_{\infty }^{\perp }-\lambda ^{\perp }||_{S^{4,2}}  \label{eqn4.7} \\
&\leq &C||2P_{0}(\lambda _{\infty }^{\perp }-\lambda ^{\perp })||_{2}  \notag
\\
&=&C||2P_{0}\lambda ^{\perp }+Q_{0}^{\perp }||_{2}  \notag
\end{eqnarray}%
by the subelliptic property of $P_{0}$ on $(\mathrm{ker}{P_{0}})^{\perp }$
and $0$ $=$ $2P_{0}\lambda _{\infty }^{\perp }+Q_{0}^{\perp }$. We compute%
\begin{equation}
\begin{array}{l}
\mid \mathcal{E}(\lambda _{\infty }^{\perp })-\mathcal{E}(\lambda ^{\perp
})\mid \\
=\left\vert \displaystyle\int_{0}^{1}\frac{d}{ds}\mathcal{E(}\lambda ^{\perp
}+s(\lambda _{\infty }^{\perp }-\lambda ^{\perp }))ds\right\vert \\[8pt]
=\left\vert \displaystyle\int_{0}^{1}\displaystyle\int_{M}[2P_{0}(\lambda
^{\perp }+s(\lambda _{\infty }^{\perp }-\lambda ^{\perp }))+Q_{0}^{\perp
}]\cdot (\lambda _{\infty }^{\perp }-\lambda ^{\perp })d\mu _{0}ds\right\vert
\\[8pt]
\leq \displaystyle\int_{0}^{1}||2P_{0}(\lambda ^{\perp }+s(\lambda _{\infty
}^{\perp }-\lambda ^{\perp }))+Q_{0}^{\perp }||_{2}\,||\lambda _{\infty
}^{\perp }-\lambda ^{\perp }||_{2}ds \\
\leq C_{1}||2P_{0}\lambda ^{\perp }+Q_{0}^{\perp }||_{2}^{2}%
\end{array}
\label{eqn4.8}
\end{equation}%
by the Cauchy inequality and (\ref{eqn4.7}). Let $\vartheta $ be a number
between $0$ and $\frac{1}{2}.$ It follows from (\ref{eqn4.8}) that%
\begin{equation}
\mid \mathcal{E}(\lambda _{\infty }^{\perp })-\mathcal{E}(\lambda ^{\perp
})\mid ^{1-\vartheta }\leq C_{2}||2P_{0}\lambda ^{\perp }+Q_{0}^{\perp
}||_{2}^{2(1-\vartheta )}\leq C_{2}||2P_{0}\lambda ^{\perp }+Q_{0}^{\perp
}||_{2}  \label{eqn4.9}
\end{equation}%
for $t$ large by noting that $2(1-\vartheta )$ $>$ $1$ and $||2P_{0}\lambda
^{\perp }+Q_{0}^{\perp }||_{2}$ tends to $0$ as $t\rightarrow \infty $ by (%
\ref{eqn4.6}). Next we compute%
\begin{equation}
\begin{array}{l}
\displaystyle-\frac{d}{dt}(\mathcal{E}(\lambda ^{\perp })-\mathcal{E}%
(\lambda _{\infty }^{\perp }))^{\vartheta } \\
=\displaystyle-\vartheta (\mathcal{E}(\lambda ^{\perp })-\mathcal{E}(\lambda
_{\infty }^{\perp }))^{\vartheta -1}\frac{d}{dt}(\mathcal{E}(\lambda ^{\perp
})-\mathcal{E}(\lambda _{\infty }^{\perp })) \\
=+\vartheta (\mathcal{E}(\lambda ^{\perp })-\mathcal{E}(\lambda _{\infty
}^{\perp }))^{\vartheta -1}||2P_{0}\lambda ^{\perp }+Q_{0}^{\perp }||_{2}||%
\dot{\lambda}^{\perp }||_{2} \\
\geq \vartheta C_{2}^{-1}||\dot{\lambda}^{\perp }||_{2}%
\end{array}
\label{eqn4.10}
\end{equation}%
by (\ref{2d}), (\ref{eqn4.9}), and noting that $\dot{\lambda}^{\perp }$ $=$ $%
-(2P_{0}\lambda ^{\perp }+Q_{0}^{\perp })$ (see (\ref{11})) and hence the
left side is nonnegative. We learned the above trick of raising the power to
$\vartheta $ from \cite{s}. Integrating (\ref{eqn4.10}) with respect to $t$
gives

\begin{equation}
\displaystyle \int_{0}^{\infty }||\dot{\lambda}^{\perp }||_{2}dt<+\infty .
\label{eqn4.11}
\end{equation}

Observing that $\frac{d}{dt}(\lambda ^{\perp }-\lambda _{\infty }^{\perp })$
$=$ $\dot{\lambda}^{\perp }$ and $-\frac{d}{dt}||\lambda ^{\perp }-\lambda
_{\infty }^{\perp }||_{2}^{2}$ $\leq $ $C||\dot{\lambda}^{\perp }||_{2},$ we
can then deduce
\begin{equation}
\lim_{t\rightarrow \infty }||\lambda ^{\perp }-\lambda _{\infty }^{\perp
}||_{2}^{2}=0  \label{eqn4.11'}
\end{equation}%
by (\ref{eqn4.11}). On the other hand, we have estimate $|r(t)|$ $\leq $ $%
C||2P_{0}\lambda ^{\perp }+Q_{0}^{\perp }||_{2}$ $=$ $C||\dot{\lambda}%
^{\perp }||_{2}$. It follows from (\ref{eqn4.11}) that

\begin{equation}
\displaystyle\int_{0}^{\infty }|r(t)|dt<+\infty .  \label{eqn4.12}
\end{equation}%
So in view of (\ref{1.6}) and (\ref{eqn4.12}), $\lambda _{\ker }$ converges
to $(\lambda _{\infty })_{\ker }$ $=$ ($\lambda _{\ker })_{\infty }$ as $t$ $%
\rightarrow $ $\infty $ (not just a sequence of times). Since $\lambda
_{\ker }-(\lambda _{\infty })_{\ker }$ $=$ $-\displaystyle\int_{t}^{\infty
}r(t)dt$ is a function of time only, we also have%
\begin{equation}
\lim_{t\rightarrow \infty }||\lambda _{\ker }-(\lambda _{\infty })_{\ker
}||_{S^{k,2}}=0  \label{eqn4.13}
\end{equation}%
for any nonnegative integer $k.$ With $\lambda ^{\perp }$ replaced by $%
\lambda ^{\perp }-\lambda _{\infty }^{\perp }$ in the argument to deduce (%
\ref{eqn1.7'}), we obtain%
\begin{equation}
\frac{d}{dt}\displaystyle\int_{M}(P_{0}^{k}(\lambda ^{\perp }-\lambda
_{\infty }^{\perp }))^{2}d\mu _{0}\leq -4\Upsilon \displaystyle%
\int_{M}(P_{0}^{k}(\lambda ^{\perp }-\lambda _{\infty }^{\perp }))^{2}d\mu
_{0}+C(k)||\lambda ^{\perp }-\lambda _{\infty }^{\perp }||_{2}.
\label{eqn4.14}
\end{equation}%
By Lemma \ref{ODE} and (\ref{eqn4.11'}), we get\
\begin{equation}
\lim_{t\rightarrow \infty }\displaystyle\int_{M}(P_{0}^{k}(\lambda ^{\perp
}-\lambda _{\infty }^{\perp }))^{2}d\mu _{0}=0.  \label{eqn4.15}
\end{equation}%
Hence by the subelliptic property of $P_{0}^{k},$ we conclude from (\ref%
{eqn4.15}) and (\ref{eqn4.11'}) that%
\begin{equation*}
\lim_{t\rightarrow \infty }||\lambda ^{\perp }-\lambda _{\infty }^{\perp
}||_{S^{4k,2}}=0.
\end{equation*}
Together with (\ref{eqn4.13}), we have proved that $\lambda $ converges to $%
\lambda _{\infty }$ smoothly as $t$ $\rightarrow $ $\infty .$%
\endproof%
\bigskip

\textbf{Proof of Main Theorem \ref{t1} : \ It follows easily from Theorem %
\ref{T61} and Theorem \ref{t32}.}

\bigskip

\end{document}